\documentclass[12pt]{amsart}
\usepackage{amssymb,amsmath,amsthm}
\usepackage[dvips]{color}

\def\x{{\boldsymbol{x}}}
\def\e{{\rm e}}
\def\eps{\varepsilon}

\def \di{{\boldsymbol{\delta}}}
\def\d{{\rm d}}
\def\dist{{\rm dist}}

\def\ddt{\frac{\d}{\d t}}
\def\R {\mathbb{R}}
\def\N {\mathbb{N}}
\def\HH{{\rm H}}
\def\H {{\mathcal H}}
\def\Ht {{\mathcal H_t}}

\def\B {{B}}
\def\Co {{\mathcal C}}

\def\BB {{\mathbb B}}
\def\D {{\rm dom}}

\def\E {{\mathcal E}}

\def\KK {{\mathbb{K}}}
\def\A {{ A}}
\def\AA{{\mathfrak A}}
\def\K {{K}}

\def\L {{\mathcal L}}
\def\Q {{\mathcal Q}}
\def \U {{\mathcal O}}

\def\X {{X}}
\def\Y {{\boldsymbol{\Sigma}}}

\def \l {\langle}
\def \r {\rangle}

\def \and{\quad\text{and}\quad}

\def \au {\rm}
\def \ti {\it}
\def \jou {\rm}
\def \bk {\it}
\def \no#1#2#3 {{\bf #1} (#3), #2.}
\def \eds#1#2#3 {#1, #2, #3.}
\newtheorem{proposition}{Proposition}[section]
\newtheorem{theorem}[proposition]{Theorem}
\newtheorem{corollary}[proposition]{Corollary}
\newtheorem{lemma}[proposition]{Lemma}
\theoremstyle{definition}
\newtheorem{definition}[proposition]{Definition}
\newtheorem{remark}[proposition]{Remark}

\numberwithin{equation}{section}
\title[Attractors for processes on time-dependent spaces]
{Attractors for processes on time-dependent spaces.\\
Applications to wave equations}

\author[M. Conti, V. Pata and R. Temam]
{Monica Conti, Vittorino Pata and Roger Temam}

\address{Politecnico di Milano - Dipartimento di Matematica ``F.\ Brioschi''
\newline\indent
Via Bonardi 9, 20133 Milano, Italy}
\email{monica.conti@polimi.it}
\email{vittorino.pata@polimi.it}

\address{Indiana University - Institute for Scientific Computing and Applied Mathematics
\newline\indent
Rawles Hall, Bloomington, IN 47405, USA}
\email{temam@indiana.edu}

\begin{document}

\begin{abstract}
For a process $U(t,\tau) : X_\tau\to \X_t$ acting on a one-parameter family of normed spaces, we
present a notion of time-dependent attractor based only on the minimality with respect to
the pullback attraction property. Such an attractor is shown to be invariant whenever
$U(t,\tau)$ is $T$-closed for some $T>0$, a much weaker property than continuity (defined in the text).
As a byproduct,
we generalize the recent theory of attractors in time-dependent spaces developed in \cite{oscillon}.
Finally, we exploit the new framework to study the longterm behavior
of wave equations with time-dependent speed of propagation.
\end{abstract}

\maketitle

%%%%%%%%%%%%%%%%%%%%%%%%%%%%%%%%%%%%%%%%%%%%%%%%%
\section{Introduction}

\noindent
The evolution of systems arising from mechanics and
physics is described in many instances by
differential equations
of the form
$$
\begin{cases}
u_t = A(u, t),\quad  t>\tau,\\
u(\tau)=u_\tau\in X,
\end{cases}
$$
where $X$ is a normed space and, for every fixed $t$, $A(\cdot,t)$ is a densely defined operator on $X$.
Assuming the Cauchy problem well posed and calling $u(t)$  the solution at time $t$,
we can construct the family of solving
operators
$$
U(t, \tau): X \to X,\quad t\geq \tau\in\R,
$$
by setting
$$U(t, \tau)u_\tau=u(t).$$
Such a family is called a \emph{process},
characterized by the properties that $U(\tau,\tau)=I$ and
$$ U(t,s)U(s, \tau) = U(t,\tau), \quad \forall t\geq s \geq \tau\in\R.
$$
The issue of understanding the longtime behavior of solutions to
dynamical systems is thus translated into studying the dissipative
properties of the operators $U(t,\tau)$.
A well-established theory of attractors
provides nowadays a
full description of many important autonomous systems from mathematical physics,
including nonautonomous models with
time-dependent external forces
(see e.g.\ the classical textbooks \cite{BV,HAL,HAR,TEM}
and the more recent references \cite{CV,MZ,ROB}).
A paradigmatic example is given by the nonlinear damped wave equation
in a smooth bounded domain $\Omega\subset\R^3$
\begin{equation}
\label{MP}
\begin{cases}
\eps u_{tt}(\x,t)+u_t(\x,t)-\Delta u(\x,t)+f(u(\x,t))=g(\x,t),\\
u(\x,t)|_{\x\in\partial\Omega}=0,
\end{cases}
\end{equation}
where $\eps>0$, $f$ is a nonlinear term and $g$ an external given force.
If $g$ is independent of time, the system is autonomous
and the problem is completely understood within the framework
of semigroups, whereas the dependence of $g$ on time
requires further integrability assumptions and the theory
of attractors for processes, suitable to deal with nonautonomous
situations.

On the contrary, the standard theory generally fails to capture the
dissipation mechanism involved in evolution problems where the coefficients
of the differential operator depend explicitly on time, leading to
time-dependent terms at a \emph{functional level}.
This can be seen in the model equation \eqref{MP},
assuming that $\eps$ is not a constant, but rather
a positive decreasing function of time $\eps(t)$ vanishing at infinity.
In such a case, the natural (twice the) energy associated to the system reads
$$\E(t)=\int_\Omega |\nabla u(\x,t)|^2\,\d \x+\eps(t)\int_\Omega|u_t(\x,t)|^2\,\d \x,$$
which exhibits a structural dependence on time.
It is then easy to realize that the vanishing character of $\eps$
at infinity alters the dissipativity of the system
and prevents the existence of absorbing sets in the usual sense,
namely, bounded sets of the phase space
$X=H^1_0(\Omega)\times L^2(\Omega)$
absorbing all the trajectories after a certain period of time.

The first two authors believe that an essential progress in this direction has been made recently
in \cite{oscillon}, where the authors adopt the new point of view
of describing the solution operator as
a family of maps
$$U(t,\tau):X_\tau\to X_t,\quad t\geq \tau\in\R,$$
acting on a time-dependent family of spaces $X_t$.
For instance, in the model problem \eqref{MP}
all the spaces coincide with the linear space $X$, but the $X_t$-norm is dictated by
the time-dependent energy $\E(t)$ of the solution at time $t$.
Based on this idea, the paper \cite{oscillon} provides
a suitable modification of the notion of pullback attractor,
establishing a new theory of pullback flavor for
dynamical systems acting on time-dependent spaces.

\subsection*{Plan of the paper}
Our aim in this article is twofold: first, in the spirit of \cite{minimalia} (and \cite{oscillon,oscillon2}),
we give new insights on attractors on time-dependent spaces.
The main idea is to define the basic objects of the theory
(such as pullback absorbing and attracting sets, time-dependent attractors)
only in terms of their attraction properties. In particular,
the time-dependent attractor will be the smallest (pullback) attracting set,
which in turn implies its uniqueness.
Quite interestingly, here neither the  process is required to be continuous, nor
the attractor to be invariant by definition.
Indeed, we prove that the invariance property is automatically satisfied
by the attractor whenever the process $U(t,\tau)$ is $T$-closed for some $T>0$,
a much weaker condition than continuity (see Definition \ref{d-semicolsed} below).
As a byproduct, we recover and improve the results of \cite{oscillon}.
The second goal is to study the longtime dynamics of the
model problem \eqref{MP} with a time-dependent coefficient $\eps(t)$.
This is done in the last part of the paper, where,
by handling the system within the new framework,
we show the existence of a time-dependent attractor of optimal regularity.
%%%%%%%%%%%%%%%%%%%%%%%%%%%%%%%%%%%%%%%%%%%%%%%%%

%%%%%%%%%%%%%%%%%%%%%%%%%%%%%%%%%%%%%%%%%%%%%%%%%
\section{The Abstract Framework}

\noindent
For $t\in\R$, let $\X_t$ be a family of normed spaces without, so far, any other hypotheses on these
spaces.
We consider a two-parameter family of operators
$$U(t,\tau):\X_\tau\to\X_t,$$
depending on $t\geq\tau\in\R$, and
satisfying the following properties:
\begin{itemize}
\item[(i)] $U(\tau,\tau)$ is the identity map on ${\X_\tau}$;
\smallskip
\item[(ii)] $U(t,\tau)U(\tau,\sigma)=U(t,\sigma)$ for every $\sigma\in\R$ and every $t\geq\tau\geq\sigma$.
\smallskip
\end{itemize}
The family $U(t,\tau)$ will still be called a process.

\begin{remark}
We stress that the spaces $\X_t$ can be in principle
completely unrelated.
Besides, no
continuity property is assumed in the definition of the process.
\end{remark}

In the next sections we will provide an abstract setting in order to study
the asymptotic behavior of the operators $U(t,\tau)$ when $t\to+\infty$ and/or $\tau\to-\infty$.
The goal is to define a suitably ``thin"
object $\AA=\{A_t\}_{t\in\R}$, where each $A_t\subset \X_t$ is able to attract
(at time $t$) all the solutions of the system originating sufficiently far in the past.
This will be done in the spirit of \cite{minimalia}, leading to the
notion of time-dependent attractor in Definition \ref{d-attractor} below.
Then, in Section \ref{sec-defs}, we  state the main existence result for time-dependent attractors,
and in the subsequent Section \ref{sec-inv} we discuss the issue of their invariance.
Finally, in Section \ref{sec-further},
we complete the presentation with some comments and a comparison with the theory
of \cite{oscillon}.

\subsection*{Notation}
For every $t\in\R$, we introduce the $R$-ball of $X_t$
$$\BB_t(R)=\big\{z\in\X_t:\,\|z\|_{\X_t}\leq R\big\}.$$
For any given $\eps>0$,
the $\eps$-neighborhood of a set $B\subset \X_t$ is defined as
$$\U^\eps_t(B)=\bigcup_{x\in B}\,\big\{y\in \X_t:\, \|x-y\|_{\X_t}<\eps\big\}
=\bigcup_{x\in B}\,\big\{x+\BB_t(\eps)\big\}.$$
We denote the Hausdorff semidistance
of two (nonempty) sets $B,C\subset \X_t$ by
$$\di_t(B,C)=\sup_{x\in B}\,\dist_{\X_t}(x,C)=\sup_{x\in B}\,\inf_{y\in C}\,\|x-y\|_{\X_t}.$$
Finally, given any set $B\subset \X_t$, the symbol $\overline{B}$ stands for the closure of $B$ in $X_t$.
%%%%%%%%%%%%%%%%%%%%%%%%%%%%%%%%%%%%%%%%%%%%%%%%%

%%%%%%%%%%%%%%%%%%%%%%%%%%%%%%%%%%%%%%%%%%%%%%%%%
\section{Pullback Attracting Sets}
\label{sec-defs}

\noindent
We begin with some definitions.

\begin{definition}
\label{d-ubdd}
A family $\mathfrak C=\{C_t\}_{t\in\R}$ of bounded sets $C_t\subset\X_t$ is called {\it uniformly bounded}
if there exists $R>0$ such that
$$C_t\subset \BB_t(R),\quad\forall t\in\R.$$
\end{definition}

\begin{definition}
\label{d-absorbing}
A family $\mathfrak B=\{\B_t\}_{t\in\R}$ is called {\it pullback absorbing}
if it is uniformly bounded and, for every $R>0$, there exists $t_0=t_0(t,R)\leq t$ such that
\begin{equation}
\label{entering}
\tau\leq t_0
\quad\Rightarrow\quad
U(t,\tau)\BB_\tau(R)\subset \B_t.
\end{equation}
The process $U(t,\tau)$ is called \emph{dissipative} whenever
it admits a pullback absorbing family.
\end{definition}

\begin{definition}
\label{d-attraction}
A (uniformly bounded) family $\mathfrak K=\{\K_t\}_{t\in\R}$ is called
\emph{pullback attracting} if
for all $\eps>0$ the family
$\{\U^\eps_t(\K_t)\}_{t\in\R}$ is pullback absorbing.
\end{definition}

\begin{remark}
\label{rem-attr}
The attracting property  can be equivalently stated in terms of the Hausdorff semidistance:
$\mathfrak K=\{\K_t\}_{t\in\R}$ is pullback attracting if and only if it is uniformly bounded and
the limit
$$\lim_{\tau\to-\infty}\,\di_t(U(t,\tau) C_\tau,\K_t)=0$$
holds for every uniformly bounded family $\mathfrak C=\{C_t\}_{t\in\R}$ and every $t\in\R$.
\end{remark}

We can describe the pullback attraction in term of sequences. To this aim,
let $\Y_t$ denote  the collection of all possible sequences of the form
$$y_n=U(t,\tau_n)x_n,$$
where $\tau_n\to-\infty$ and $x_n\in \X_{\tau_n}$ is any uniformly bounded sequence. For any $y_n\in\Y_t$ we denote
$$\L_t(y_n)=\big\{x\in \X_t:\, \|y_{n_\imath}- x\|_{\X_t}\to 0 \,\text{
for some subsequence $n_\imath\to\infty$} \big\}.
$$
It is immediately seen from the definitions that a uniformly bounded family
$\mathfrak K=\{\K_t\}_{t\in\R}$ is pullback attracting if and only if
\begin{equation}
\label{l-sca}
\dist_{\X_t} (y_n,\K_t)\to 0,\quad\forall y_n\in \Y_t,
\end{equation}
for all $t\in \R$.
In particular, each element of $\L_t(y_n)$ belongs to the closure
of $\K_t$. Therefore, setting
$$\A^\star_t=\bigcup_{y_n\in\Y_t}\L_t(y_n),$$
we have proved

\begin{lemma}
\label{l-mini}
Assume that there exists a pullback attracting family
of closed sets $\mathfrak K=\{\K_t\}_{t\in\R}$. Then
$$\A^\star_t\subset \K_t,\quad\forall t\in\R.$$
\end{lemma}

\begin{lemma}
\label{l-astar}
If the process $U(t,\tau)$ is dissipative, then ${\mathfrak A}^\star=\{A^\star_t\}_{t\in\R}$ coincides
with the time-dependent $\omega$-limit of any pullback absorbing set $\mathfrak B=\{\B_t\}_{t\in\R}$,
that is:
\begin{equation}
\label{e-astar}
A^\star_t=
\bigcap_{y\leq t}\,\overline{\bigcup_{\tau\leq y}\,U(t,\tau)\B_\tau}.
\end{equation}
In particular, $A^\star_t$ is closed and contained in $\overline{\B_t}$ for all $t\in\R$;
hence ${\mathfrak A}^\star$ is uniformly bounded.
\end{lemma}

\begin{proof}
The validity of \eqref{e-astar} is a direct consequence of the definitions, so $A^\star_t$
is obviously closed.
Besides, since $\mathfrak B$ is uniformly bounded, it absorbs itself and
$$U(t,\tau)\B_\tau\subset \B_t,\quad \forall \tau\leq t_0,$$
for some $t_0=t_0(t,{\mathfrak B})\leq t$,
implying the inclusion $A^\star_t\subset \overline{\B_t}$.
\end{proof}

The next lemma characterizes the attraction property for compact sets.

\begin{lemma}
\label{UNO}
Let $\mathfrak K=\{\K_t\}_{t\in\R}$ be a uniformly bounded family
of compact sets. Then $\mathfrak K$ is pullback attracting if and only if for all $t\in\R$
$$\emptyset\neq \L_t(y_n)\subset \K_t,\quad \forall y_n\in\Y_t.$$
\end{lemma}

\begin{proof}
Let $\mathfrak K=\{\K_t\}_{t\in\R}$ be a family of compact sets.
If $\mathfrak K$ is pullback attracting, then given $y_n\in \Y_t$ we have
$$\L_t(y_n)\subset \K_t\quad\text{and}\quad \dist_{\X_t} (y_n,\xi_n)\to 0,$$
for some $\xi_n\in \K_t$. Since $\K_t$ is compact, there exists $\xi\in \K_t$
such that (up to a subsequence)
$$\xi_n\to\xi\in \K_t\quad\Rightarrow\quad y_n\to\xi\quad\Rightarrow\quad \L_t(y_n)\neq\emptyset.$$
Conversely,
if $\mathfrak K$ is not pullback attracting, we deduce from \eqref{l-sca} that
$$\dist_{\X_t}(y_n,\K_t)>\eps,$$
for some $t\in\R$, $\eps>0$ and $y_n\in\Y_t$. Therefore, $\L_t(y_n)\cap \K_t=\emptyset$.
\end{proof}
%%%%%%%%%%%%%%%%%%%%%%%%%%%%%%%%%%%%%%%%%%%%%%%%%

%%%%%%%%%%%%%%%%%%%%%%%%%%%%%%%%%%%%%%%%%%%%%%%%%
\section{Time-Dependent Global Attractors}

\noindent
It is clear from the earlier discussion that a pullback
attracting family of compact sets is capable of controlling the regime
of the system at any time $t\in\R$. This leads quite naturally to the definition
of an attractor as the smallest set possessing such a property. To this aim we consider
the collection
\begin{equation}
\label{setk}
\KK=\big\{\mathfrak K=\{\K_t\}_{t\in\R}\,: \K_t\subset \X_t \text{ compact,}\;
\mathfrak K \text{ pullback attracting}\big\}.
\end{equation}

\smallskip
\noindent
When $\KK\neq \emptyset$ we say that the process is \emph{asymptotically compact}.

\begin{definition}
\label{d-attractor}
We call a \emph{time-dependent global attractor} the smallest element of $\KK$, i.e. the family
$\AA=\{\A_t\}_{t\in\R}\in\KK$ such that
$$\A_t\subset \K_t,\quad\forall t\in\R,$$
for any element $\mathfrak K=\{\K_t\}_{t\in\R}\in\KK$.
\end{definition}

The next result tells that the definition is consistent: the minimal element of $\KK$
exists (and it is unique) if and only if $\KK$ is not empty.

\begin{theorem}
\label{MAA}
If $U(t,\tau)$ is asymptotically compact, then the time-dependent
attractor $\AA$ exists and coincides with the set $\AA^\star=\{\A^\star_t\}_{t\in\R}$.
In particular, it is unique.
\end{theorem}

\begin{proof}
Let $\mathfrak K=\{\K_t\}_{t\in\R}$ be an element of $\KK$.
Then, by Lemmas \ref{l-mini} and \ref{UNO},
$$\emptyset\neq \A^\star_t\subset K_t,\quad \forall t\in\R.$$
Since $U(t,\tau)$ is dissipative, we know by Lemma \ref{l-astar}
that $\AA^\star$ is uniformly bounded and
$\A^\star_t$ is closed for all $t\in\R$.
Since $\A^\star_t$ is contained in the compact set $K_t$, then $\A^\star_t$ is compact as well.
The attraction property is contained in \eqref{l-sca}, saying that
$\AA^*$ is an element of $\KK$. Thanks to Lemma \ref{l-mini}
it is also the smallest element of $\KK$, hence it is the (unique)
time-dependent attractor by the very definition.
\end{proof}

We now provide a necessary condition for $\KK$ to be nonempty, which turns out to be sufficient
as well when the spaces $\X_t$ are complete.

\begin{definition}
\label{d-efone}
A process  $U(t,\tau)$ is {\it $\eps$-dissipative} if for every $t\in\R$ there exists a
set $F_t\subset\X_t$ made of a finite number of points such that
the family $\{\U_t^\eps(F_t)\}_{t\in\R}$ is pullback absorbing (cf.\ Definition 4 in \cite{minimalia}).
The process is called {\it totally dissipative} whenever
it is $\eps$-dissipative for every $\eps>0$. Note that the sets $F_t$ need not be the same
for all $\eps$.
\end{definition}

\begin{theorem}
\label{suffy}
Assume that $\X_t$ is a Banach space for all $t\in\R$.
Then $U(t,\tau)$ is totally dissipative if and only if $\KK\neq \emptyset$.
\end{theorem}

\begin{proof}
If $\KK\neq\emptyset$, then $U(t,\tau)$ is totally dissipative.
Indeed, if $\mathfrak K=\{\K_t\}_{t\in\R}$ belongs to $\KK$, it follows that
any $\K_t$ can be covered by finitely many $\eps$-balls, and calling $F_t$ the union of the
centers of those balls, the family
$\{\U_t^\eps(F_t)\}_{t\in\R}$ is pullback absorbing.
Conversely, if $U(t,\tau)$ is totally dissipative, for any arbitrarily fixed $\eps>0$,
we can choose a finite set $F^\eps_t$
such that the family $\{\U^\eps_t(F^\eps_t)\}_{t\in\R}$ is uniformly bounded and absorbing.
If we select any $y_n\in\Y_t$, then  $y_n$ eventually falls into
$$V^\eps_t=\overline{\U^\eps_t(F^\eps_t)}.$$
Set
$$\K_t=\bigcap_{\eps>0}V^\eps_t.$$
Accordingly, the family
$\mathfrak K=\{\K_t\}_{t\in\R}$ is uniformly bounded.
Furthermore, both $\K_t$ and $\{y_n\}$ are coverable by finitely many
balls of arbitrarily small radius, which, in Banach spaces, means
precompactness.
In particular, $\K_t$ being closed, it is compact in $\X_t$.
Since the sequence $y_n$ is precompact, then $\L_t(y_n)$ is nonempty. Also, it is
contained in every closed set $V^\eps_t$ and hence in their intersection $\K_t$.
In other words,
$$\dist_{\X_t}( y_n, \K_t)\to 0,$$
meaning that $\mathfrak K$ is pullback attracting.
Hence $\mathfrak K\in\KK$.
\end{proof}

Collecting Theorem \ref{MAA} and \ref{suffy} we draw a corollary.

\begin{corollary}
If the family $U(t,\tau)$ is totally dissipative, then the time-dependent
attractor $\AA$ exists and coincides with the set
$\AA^\star$. In particular, it is unique and uniformly bounded.
\end{corollary}

\begin{remark}
A less direct
characterization of a totally dissipative process is also possible,
based on the {\it Kuratowski measure of noncompactness}
of a bounded set $B\subset \X_t$ (see \cite{HAL})
$$\boldsymbol{\alpha}_t(B)=\inf\big\{d>0: \text{$B$ has a finite covering by balls of
$\X_t$ of diameter less than $d$}\big\}.$$
Indeed, it is easily seen that the family
$U(t,\tau)$ is totally dissipative if and only if
there exists a
pullback absorbing set $\mathfrak B=\{\B_t\}_{t\in\R}$
for which
$$\lim_{\tau\to-\infty}\boldsymbol{\alpha}_t (U(t,\tau)\B_\tau)=0,\quad \forall t\in\R.
$$
\end{remark}
%%%%%%%%%%%%%%%%%%%%%%%%%%%%%%%%%%%%%%%%%%%%%%%%%

%%%%%%%%%%%%%%%%%%%%%%%%%%%%%%%%%%%%%%%%%%%%%%%%%
\section{Invariance of the Attractor}
\label{sec-inv}

\noindent
A further question is the invariance of the time-dependent global attractor.

\begin{definition}
We say that $\AA=\{\A_t\}_{t\in\R}$ is \emph{invariant} if
$$U(t,\tau)\A_\tau=\A_t, \quad \forall t\geq \tau.$$
\end{definition}

This property is usually a priori postulated in the literature. In particular, in \cite{oscillon}
the time-dependent attractor is by definition a family of compact
sets which is at the same time pullback attracting
and invariant, and its existence is proved by exploiting
the continuity of the process $U(t,\tau)$ (see
Theorem 2.1 and Remark 2.4 therein).

\begin{remark}
On the other hand, if we know that $\mathfrak K$ is an invariant pullback attracting family of compact sets,
it is clear that $\mathfrak K$ is the smallest element of $\KK$, hence it coincides with
the time-dependent attractor $\AA$.
\end{remark}

Our purpose here is to show that the time-dependent global attractor
provided by Theorem \ref{MAA} is automatically invariant whenever the process $U(t,\tau)$ is $T$-closed
for some $T>0$
in the sense of Definition \ref{d-semicolsed} below, a very mild continuity-like assumption.
We start with a sufficient condition.

\begin{proposition}
\label{propINV}
If there exists $T>0$ such that
$$\A_t\subset U(t,t-T)\A_{t-T},\quad
\forall t\in\R,$$
then $\AA$ is invariant.
\end{proposition}

\begin{proof}
Let $t\in\R$ be arbitrarily fixed.
For any $s\geq t$ and any $n\in\N$, we have by induction
\begin{equation}
\label{Uuno}
U(s,t) \A_t\subset U(s,t-T)\A_{t-T}\subset\cdots
\subset U(s,t-nT)\A_{t-nT}.
\end{equation}
Consequently,
$$\di_{s}(U(s,t)\A_t,\A_s)\leq \di_{s}(U(s,t-nT)\A_{t-nT},\A_s).
$$
Since $\AA$ is attracting, letting $n\to\infty$ we obtain
$$\di_{s}(U(s,t)\A_t,\A_s)=0,
$$
implying in turn, since $\A_s$ is closed, that
\begin{equation}
\label{Udue}
U(s,t)\A_t\subset \A_s,\quad\forall s\geq t.
\end{equation}
In particular, \eqref{Uuno}-\eqref{Udue}
for $s=t$ entail
$$\A_t\subset U(t,t-nT)\A_{t-nT}\subset \A_t,
$$
that is,
\begin{equation}
\label{Utre}
\A_t=U(t,t-nT)\A_{t-nT}.
\end{equation}
Let now $\tau\leq t$. Taking $n$ large enough, we infer from \eqref{Udue}-\eqref{Utre} that
$$\A_t=U(t,t-nT)\A_{t-nT}=U(t,\tau)U(\tau,t-nT)\A_{t-nT}\subset U(t,\tau)\A_\tau\subset \A_t,$$
proving the equality $U(t,\tau)\A_\tau=\A_t$.
\end{proof}

In order to establish an invariance criterion, we need one more definition.
Recall that, for any pair of fixed times $t\geq \tau$, the map $U(t,\tau) : \X_\tau\to \X_t$
is said to be closed if
$$
\begin{cases}
x_n\to x & \text{in }\X_\tau\\
U(t,\tau)x_n\to\zeta & \text{in }\X_t
\end{cases}\quad\Rightarrow\quad U(t,\tau)x=\zeta.
$$

\begin{definition}
\label{d-semicolsed}
The process
$U(t,\tau)$ is called
\begin{itemize}
\smallskip
\item[$\bullet$] \emph{closed} if
$U(t,\tau)$ is a closed map for any pair of fixed times $t\geq \tau$;
\smallskip
\item[$\bullet$] \emph{$T$-closed} for some $T>0$ if
$U(t,t-T)$ is a closed map for all $t$.
\end{itemize}
\end{definition}

\begin{remark}
Of course if the process  $U(t,\tau)$ is closed it is $T$-closed, for any $T>0$.
Note also that if the process $U(t,\tau)$ is a continuous (or even norm-to-weak continuous) map for all $t\geq \tau$,
then the process is closed.
\end{remark}

\begin{theorem}
\label{invariance-2}
If $U(t,\tau)$ is a $T$-closed process for some $T>0$, which possesses a time-dependent global attractor
$\AA$,
then $\AA$ is invariant.
\end{theorem}

\begin{proof}
In view of Proposition~\ref{propINV},
it is enough to prove the inclusion
$$\A_t\subset U(t,t-T)\A_{t-T},\quad
\forall t\in\R.$$
To this end, select an arbitrary $y\in \A_t$. By Theorem~\ref{MAA},
$$y_n\to y\quad\text{for some}\quad y_n=U(t,\tau_n)x_n\in\Y_t.$$
Define the  sequence
$$w_n=U(t-T,\tau_n)x_n.$$
On account of
Lemma~\ref{UNO},
$$w_n\to w\quad\text{for some}\quad w\in \A_{t-T}.$$
On the other hand,
$$U(t,t-T)w_n=U(t,\tau_n)x_n=y_n\to y,
$$
and since $U(t,t-T)$ is closed we conclude that
$$U(t,t-T)w=y.$$
Therefore
$$y\in U(t,t-T)\A_{t-T},$$
yielding the desired inclusion.
\end{proof}

\begin{remark}
In fact, Theorem \ref{invariance-2} holds under a weaker continuity condition on $U(t,\tau)$.
It suffices to require the existence of
sequence
$$0=T_0<T_1<T_2<T_3\ldots \to\infty$$
with the following property: for all $k\in\N$
$$
\begin{cases}
x_n^k\to\xi_0^k & \text{in } X_{t-T_k}\\
U(t,t-T_k)x_n^k\to \xi^k & \text{in }\X_t
\end{cases}\quad\Rightarrow\quad U(t,t-T_k)\xi^k_0=\xi^k.
$$
If so, we call the process {\it asymptotically closed},
in analogy to the semigroup case
(see \cite{minimalia}).
\end{remark}
%%%%%%%%%%%%%%%%%%%%%%%%%%%%%%%%%%%%%%%%%%%%%%%%%

%%%%%%%%%%%%%%%%%%%%%%%%%%%%%%%%%%%%%%%%%%%%%%%%%
\section{Further Remarks}
\label{sec-further}

\subsection*{I} The notion of pullback absorber given in \cite{oscillon}
looks apparently different from
ours, and is based on the notion of a \emph{pullback-bounded family}, namely, a
family $\mathfrak B=\{\B_t\}_{t\in\R}$ satisfying
\begin{equation}
\label{RRR}
R(t)=\sup_{\tau\leq t}\|\B_\tau\|_{\X_t}<\infty,\quad \forall t\in\R.
\end{equation}
Accordingly, $\mathfrak B$ is called a
\emph{pullback absorber} if it is a pullback-bounded family with the following property: for every
$t\in\R$ and every pullback-bounded family $\mathfrak C=\{C_t\}_{t\in\R}$ there exists
$t_0=t_0(t,\mathfrak C)\leq t$ such that
\begin{equation}
\label{entering-oscillon}
\tau\leq t_0
\quad\Rightarrow\quad
U(t,\tau)C_\tau\subset B_t.
\end{equation}
Since any family of balls $\{\BB_t(R)\}_{t\in\R}$ is pullback-bounded,
\eqref{entering} obviously follows from \eqref{entering-oscillon}.
As a matter of fact, the two notions of absorbtion are equivalent.
Indeed, if $\mathfrak C$ is any pullback-bounded family with maximal size $R(t)$ on $(-\infty,t]$,
then
$$U(t,\tau)C_\tau\subset U(t,\tau)\BB_\tau(R(t)),\quad \forall \tau\leq t.$$
Hence, if $\mathfrak B$ is a pullback absorbing family in the sense of Definition \ref{d-absorbing}
and $t\in\R$ is any fixed time,
we have
$$\tau\leq t_0
\quad\Rightarrow\quad
U(t,\tau)C_\tau\subset U(t,\tau)\BB_\tau(R(t))\subset B_t
$$
for some $t\geq t_0=t_0(t,R(t))$, where $R(t)$ depends only on $\mathfrak C$. But this is exactly
the absorbtion property \eqref{entering-oscillon}.

\smallskip
\noindent
In the present work, we decided to postulate in the definition of absorbing family
the stronger property of being uniformly bounded, instead of merely pullback bounded.
Such a notion seems to reflect more closely the dissipation mechanism of most equations of
mathematical physics, where the dynamics at time $t$ is confined in bounded sets
$\B_t$ (the pullback absorbing family) whose size in the phase space $\X_t$ remains bounded as $t\to+\infty$
(whereas the boundedness as $t\to-\infty$ is a consequence of \eqref{RRR}).
This happens, for instance, for the so-called \emph{oscillon} equation
arising in recent cosmological theories that motivated the authors of \cite{oscillon}
to develop this novel theory (see also \cite{oscillon2}),
as well as for the wave equation \eqref{MP} studied in this paper.
Conversely, having a pullback bounded absorbing family does not prevent the possibility of
$\B_t$ becoming larger and larger as time increases, in contrast with the common intuition of dissipation.

\subsection*{II} Similarly, the notion of pullback attracting set
in Definition \ref{d-attraction} (or Remark \ref{rem-attr}) can be rephrased in
the following way: a uniformly bounded family $\mathfrak K=\{\K_t\}_{t\in\R}$ is
pullback attracting if and only if
\begin{equation}
\label{attr-oscillon}
\lim_{\tau\to-\infty}\,\di_t(U(t,\tau)C_t,\K_t)=0
\end{equation}
for every pullback-bounded family $\mathfrak C=\{C_t\}_{t\in \R}$ and $t\in\R$.
Observe that this is exactly the pullback attraction property defined in
\cite{oscillon}.

\subsection*{III} An interesting question is whether
property \eqref{attr-oscillon} holds uniformly with respect to intervals of time.
This is not  true in general. In particular, it cannot happen on unbounded intervals.
The next result shows that, if the process is sufficiently smooth,
then the attraction exerted by any invariant pullback attracting family
(such as the time-dependent attractor) is uniform on compact intervals.

\begin{proposition}
\label{unif-att}
Let $\mathfrak K=\{\K_t\}_{t\in\R}$ be an invariant pullback attracting family.
Assume that
\begin{equation}
\label{lip}
\|U(t,\tau)z_1-U(t,\tau)z_2\|_{\X_t}\leq \Q(t-\tau,r)\|z_1-z_2\|_{\X_\tau},
\end{equation}
for all $t\geq\tau\in\R$ and $\|z_i\|_{\H_t}\leq r$, where
$\Q$ is a positive function, increasing in each of its arguments.
Then, for all $R>0$,
$$\lim_{\tau\to-\infty}\,\di_t(U(t,\tau)\BB_\tau(R),\K_t)=0,$$
uniformly for $t$ belonging to a compact set.
\end{proposition}

\begin{proof}
Let $[a,b]$ with $-\infty<a<b<\infty$ be given.
Let $R_0>0$ be such that
$$\mathcal O^1_a(\K_a)\subset\BB_a(R_0).$$
For every $\varrho>0$ small enough, set
$$\eps=\frac{\varrho}{\Q(b-a,R_0)}<1.$$
Since $\mathfrak K$ is pullback attracting,
for any given $R>0$ there exists
$$\tau_0=\tau_0(R,\eps)<a$$
such that
$$U(a,\tau)\BB_\tau(R)\subset {\mathcal O}_{\eps}(\K_a),\quad \forall\tau<\tau_0.$$
Let now $\tau<\tau_0$ be fixed, and select any $x\in \BB_\tau(R)$.
Calling $z=U(a,\tau)x$, choose $k\in \K_a$ for which
$$\|z-k\|_{\X_a}<\eps.$$
Then, in light of \eqref{lip}, for all $t\in [a,b]$ we have
$$
\|U(t,a)z-U(t,a)k\|_{\X_t}\leq \Q(t-a,R_0)\|z-k\|_{\X_a}\leq\eps\Q(b-a,R_0)=\varrho.
$$
Observe that, from the invariance of $\mathfrak K$,
$$U(t,\tau)x=U(t,a)U(a,\tau)x=U(t,a)z\quad\text{ and }\quad U(t,a)k\subset \K_t.$$
Thus,
$$
\dist_{\X_t}(U(t,\tau)x, \K_t)\leq \|U(t,a)z-U(t,a)k\|_{\X_t}\leq\varrho.
$$
In conclusion, we proved that for all $\varrho>0$ small there exists $\tau_0<a$ such that
$$
\di_t(U(t,\tau)\BB_\tau(R), \K_t)\leq\varrho,\quad \forall\tau<\tau_0.
$$
Since $\tau_0$ is independent
of $t\in[a,b]$, the proof is finished.
\end{proof}
%%%%%%%%%%%%%%%%%%%%%%%%%%%%%%%%%%%%%%%%%%%%%%%%%%%%%%%%%%%

%%%%%%%%%%%%%%%%%%%%%%%%%%%%%%%%%%%%%%%%%%%%%%%%%%%%%%%%%%%
\section{Wave Equations with Time-Dependent Speed of Propagation}

\noindent
We now want to apply the
theory above to the nonautonomous wave equation
\eqref{EQ}-\eqref{IC} below.
Let $\Omega\subset \R^3$ be a bounded domain with smooth boundary $\partial\Omega$. For any $\tau\in \R$,
we consider the evolution equation for the unknown variable
$u=u(\x,t):\Omega\times [\tau,\infty)\to \R$
\begin{equation}
\label{EQ}
\eps u_{tt}+ \alpha u_t-\Delta u+f(u)=g,\quad t>\tau,\\
\end{equation}
subject to Dirichlet boundary condition
\begin{equation}
\label{BC}
u_{|\partial\Omega}=0,
\end{equation}
and to the initial conditions
\begin{equation}
\label{IC}
u(\x, \tau) = a(\x)\quad\text{ and }\quad u_t(\x, \tau) = b(\x),
\end{equation}
where $a,b:\Omega\to\R$ are assigned data. Here $\eps=\eps(t)$ is a function of $t$
and we postulate the following assumptions.

\subsection{Conditions on $\boldsymbol{\eps}$}
We let $\eps\in \Co^1(\R)$ be a decreasing bounded function satisfying
\begin{equation}\label{e1}
\lim_{t\to +\infty}\eps(t)=0.
\end{equation}
In particular, there exists $L>0$ such that
\begin{equation}\label{e2}
\sup_{t\in\R}\big[|\eps(t)|+|\eps'(t)|\big]\leq L.
\end{equation}

\subsection{Conditions on $\boldsymbol{f}$}
We let
$f\in \mathcal C^2(\R)$ with
$f(0)=0$ satisfying, for every $s\in\R$, the growth bound
\begin{equation}
\label{phi2}
|f''(s)|\leq c(1+|s|),\quad\text{ for some }c\geq 0,
\end{equation}
along with the dissipation condition
\begin{equation}
\label{phi1}
\liminf_{|s|\to\infty}\frac{f(s)}{s}>-\lambda_1,
\end{equation}
where $\lambda_1>0$ is the first eigenvalue of
the strictly positive Dirichlet operator
$$A=-\Delta\quad\text{with domain}\quad
\D(A)=H^2(\Omega)\cap H_0^1(\Omega)\Subset L^2(\Omega).$$
Finally, the damping coefficient $\alpha$ is a positive constant
and the time-independent external source $g = g(\x)$ is taken in $L^2(\Omega)$.

\medskip
\noindent
Equation \eqref{EQ} can be seen as a nonlinear
damped wave equation with time-dependent speed of
propagation $1/\eps(t)$. Besides, it can also be interpreted as a model for
the thermal evolution in a homogenous isotropic (rigid) heat conductor according to the
Maxwell-Cattaneo law \cite{CAT} (see also \cite[Appendix B]{LSD}), with $\eps(t)$ representing
a time-dependent relaxation parameter.

\medskip
\noindent
In the case when $\eps$ is a
positive constant, the asymptotic behavior of solutions to equation \eqref{EQ}-\eqref{IC}
has been the object of extensive studies since the eighties (see, e.g.\ \cite{ach,bv0,BV,gt}),
and it is well-known to generate
a strongly continuous semigroup $S(t)$ on the phase space
$$\H=H_0^1(\Omega)\times L^2(\Omega).$$
We refer the reader to the recent reference \cite{PZ-wdwe}
for a review on the subject and a discussion on the assumptions
on $f$ which are suitable to prove the existence of the compact global attractor
of optimal regularity.

\medskip
\noindent
The aim of the subsequent sections is to study
the longtime behavior of the solutions to \eqref{EQ}-\eqref{IC} with $\eps$ depending on time,
according to the abstract framework developed in the first part of this article.
Our main result is Theorem \ref{MAIN-WAVE} below, proving the existence of a time-dependent
global attractor for the process associated with \eqref{EQ} acting on a suitable time-dependent
family of spaces. Besides, the attractor turns out to be invariant and of optimal regularity,
in a sense explained below.
%%%%%%%%%%%%%%%%%%%%%%%%%%%%%%%%%%%%%%%%%%%%%%%%%%%%%%

%%%%%%%%%%%%%%%%%%%%%%%%%%%%%%%%%%%%%%%%%%%%%%%%%%%%%%%
\section{Preliminaries}

\subsection{The functional setting}
We set $\HH=L^2(\Omega)$, with inner product
$\l \cdot,\cdot\r$ and norm $\|\cdot\|$.
For $0\leq\sigma\leq 2$,
we define the hierarchy of (compactly) nested Hilbert spaces
$$\HH_\sigma=\D(A^{\sigma/2}),\quad
\l w,v\r_\sigma=\l A^{\sigma/2}w,A^{\sigma/2}v\r,\quad
\|w\|_\sigma=\|A^{\sigma/2}w\|.$$
Then, for $t\in\R$ and $0\leq\sigma\leq 2$, we introduce the  time-dependent spaces
$$\H^\sigma_t=\HH_{\sigma+1}\times \HH_\sigma$$
endowed with the time-dependent product norms
$$\|\{a,b\} \|^2_{\H^\sigma_t}= \|a\|^2_{\sigma+1}+\eps(t)\|b\|^2_\sigma.$$
The symbol $\sigma$ is always omitted whenever zero. In particular,
the time-dependent phase space where we settle the problem
is
$$\H_t=\HH_{1}\times \HH\quad\mbox{ with }\quad \|\{a,b\} \|^2_{\H_t}= \|a\|^2_{1}+\eps(t)\|b\|^2.$$
Then, we have the compact embeddings
$$\H^\sigma_t\Subset \H_t,\quad 0<\sigma\leq 2,
$$
with injection constants independent of $t\in\R$.
Note that the spaces $\H_t$ are all the same as
linear spaces; besides, since $\eps(\cdot)$ is a decreasing function of $t$,
for every $z\in \HH_{1}\times \HH$ and $t\geq \tau\in\R$ there holds $$
\|z\|^2_{\H_t}\leq \|z\|^2_{\H_\tau}\leq \max\Big\{1,\frac{\eps(\tau)}{\eps(t)}\Big\}\|z\|^2_{\H_t}.
$$
Hence the norms $\|\cdot\|^2_{\H_t}$ and $\|\cdot\|^2_{\H_\tau}$ are equivalent for any fixed $t,\tau\in\R$, but
the equivalence constant blows up when $t\to+\infty$.

\medskip
\noindent
Along the paper, we will perform a number of formal energy-type estimates,
which are rigorously justified in a Galerkin approximation scheme.  Moreover, the H\"older, Young and Poincar\'e
inequalities will be tacitly used.

\subsection{Technical lemmas}
We shall exploit the following Gronwall-type lemma, whose proof
 can be found in \cite{visco}.
\begin{lemma}\label{gronwall}
Let $\Lambda : [\tau,\infty)\to\R^+$ be an absolutely continuous function satisfying the inequality
$$\ddt\Lambda(t) + 2\omega\Lambda(t)\leq q(t)\Lambda(t) + k$$
for some $\omega>0, k\geq 0$ and where
$q: [\tau,\infty)\to\R^+$ fulfills
$$\int_\tau^\infty q(y)\,\d y\leq m, $$
with $m\geq 0$.
Then,
$$\Lambda(t)\leq \Lambda(\tau)\e^m\e^{-\omega (t-\tau)}+k\omega^{-1}\e^m.$$
\end{lemma}

\medskip
\noindent
Calling
$$F(s)=\int_0^sf(y)\,\d y, $$
in light of \eqref{phi1} it is a standard matter to verify that
\begin{lemma}\label{l-tec2}
The following inequalities hold for some $0<\nu<1$ and $c_1\geq 0$:
\begin{align}
\label{phi-funz1}
&2\l F(u),1\r\geq -(1-\nu)\|u\|_1^2-c_1,\\
\label{phi-funz2}
&\l f(u),u\r\geq -(1-\nu)\|u\|_1^2-c_1,\quad\forall u\in\HH_1.
\end{align}
\end{lemma}

\subsection{A word of warning}
Similarly to the classical damped wave equation with constant coefficients,
proving the dissipativity of the system with $f(u)$ satisfying
\eqref{phi2} and \eqref{phi1} is quite technical and requires
several steps (see e.g. \cite{ach,BV}, see also \cite{BP} for a different strategy).

Since the main focus in this paper is the presence of the time-dependent coefficient $\eps(t)$,
in order to avoid technical complications only due to the nonlinear term $f(u)$,
we require the additional assumption
\begin{equation}
\label{extra}
2\l f(u),u\r\geq 2\l F(u),1\r-(1-\nu)\|u\|_1^2-c_1,
\end{equation}
that will be used for providing a simple and direct proof of Theorem \ref{t-abset}
below.
Condition \eqref{extra} is ensured  by asking, for instance, that
$$
\liminf_{|s|\to\infty} f'(s)>-\lambda_1,
$$
which is slightly less general than \eqref{phi1} but still widely used in the
literature.
%%%%%%%%%%%%%%%%%%%%%%%%%%%%%%%%%%%%%%%%%%%%%%%%%

%%%%%%%%%%%%%%%%%%%%%%%%%%%%%%%%%%%%%%%%%%%%%%%%%
\section{Well-Posedness}

\noindent
For any $\tau\in\R$, we rewrite problem \eqref{EQ}-\eqref{IC} as
\begin{equation}
\label{P}
\begin{cases}
\eps u_{tt}+ \alpha u_t+Au+f(u)=g,\quad t>\tau,\\
u(\tau)=a,\\
u_t(\tau)=b.
\end{cases}
\end{equation}

\begin{theorem}
\label{wellposed}
Problem \eqref{P} generates a strongly continuous process $U(t, \tau) : \H_\tau\to\H_t$,
$t\geq \tau\in\R$,
satisfying the following  continuous dependence property: for
every pair of initial  data  $z_i=\{a_i,b_i\}\in \H_\tau$
such that $\|z_i\|_{\H_\tau}\leq R$, $i=1,2$, the difference
of the corresponding solutions satisfies
\begin{equation}
\label{contdep}
\|U(t,\tau)z_1-U(t,\tau)z_2 \|_{\H_t} \leq \e^{K(t-\tau)}\|z_1-z_2\|_{\H_\tau},\quad \forall t\geq \tau,
\end{equation}
for some constant $K = K(R) \geq 0$.
\end{theorem}
Global existence of (weak) solutions $u$
to \eqref{P} is classical, and can be obtained by means of a standard Galerkin scheme,
based on the subsequent Lemma \ref{l-energy0}.
Such solutions satisfy, on any interval $[\tau,t]$ with $t\geq \tau$,
$$u\in \mathcal C([\tau,t],\HH_1),\quad u_t\in \mathcal C([\tau,t],\HH),$$
see e.g.\ \cite{TEM}.
Uniqueness of solutions will then
follow by the continuous dependence estimate \eqref{contdep}.
As a consequence, the family of maps with $t\geq\tau\in\R$
$$U(t,\tau):\H_\tau\to\H_t\quad\mbox{ acting as } \quad U(t,\tau)z=\{(u(t), u_t(t)\},$$
where $u$ is the unique solution to \eqref{P} with initial time $\tau$ and initial condition $z=\{a,b\}\in\H_\tau$,
defines a strongly continuous process on the family $\{\H_t\}_{t\in\R}$.

\medskip
\noindent \emph{Proof of estimate \eqref{contdep}}.
Let $z_1,z_2 \in \H_\tau$ be
such that $\|z_i\|_{\H_\tau}\leq R$, $i=1,2$ and denote by $C$ a
generic positive constant depending on $R$ but independent of $z_i$.
We first observe that the energy estimate in Lemma \ref{l-energy0} below ensures
\begin{equation}
\label{bddata}
\|U(t,\tau)z_i\|_{\H_t}\leq C.
\end{equation}
We call $\{u_i(t),\partial_t u_i(t)\}=U(t,\tau)z_i$ and denote
$\bar z(t)=\{\bar u(t),\bar u_t(t)\}=U(t,\tau)z_1-U(t,\tau)z_2$.
Then, the difference between the two solutions satisfies
\begin{equation*}
\eps \bar u_{tt}+ \alpha \bar u_t-\Delta \bar u+f(u_1)-f(u_2)=0,
\end{equation*}
with initial datum $z(\tau)=z_1-z_2$.
Multiplying by $2\bar u_t$ we obtain
\begin{equation*}
\ddt \|\bar z\|_{\H_t}^2+ [2\alpha-\eps']\|\bar u_t\|^2=-2\l f(u_1)-f(u_2),\bar u_t\r.
\end{equation*}
Estimating the right-hand side in light of \eqref{phi2} and \eqref{bddata}
\begin{align*}
-2\l f(u_1)-f(u_2),\bar u_t\r\leq C\|\bar u\|_1\|\bar u_t\|\leq \frac{\alpha}{2}\|\bar u_t\|^2
+C\|\bar u\|_1^2,
\end{align*}
we end up with the differential inequality
\begin{equation*}
\ddt \|\bar z(t)\|_{\H_t}^2\leq C\|\bar z(t)\|_{\H_t}^2,
\end{equation*}
and an application of the Gronwall lemma on $[\tau,t]$ completes the proof.\hfill$\Box$
%%%%%%%%%%%%%%%%%%%%%%%%%%%%%%%%%%%%%%%%%%%%%%%%%

%%%%%%%%%%%%%%%%%%%%%%%%%%%%%%%%%%%%%%%%%%%%%%%%%
\section{Absorbing Sets}

\noindent
This section is devoted to studying the dissipation properties of the process
$U(t,\tau)$ associated with \eqref{P}. We start with a new notion of absorbtion, which is {\it stronger}
than the pullback dissipativity of Definition \ref{d-absorbing}.
\begin{definition} A {\it time-dependent absorbing set} for the process $U(t,\tau)$ is
a uniformly bounded family $\mathfrak B=\{\B_t\}_{t\in\R}$
with the following property: for every $R\geq 0$ there exists $\theta_\e=\theta_\e(R)\geq0$ such that
$$\tau\leq t-\theta_\e
\quad\Rightarrow\quad
U(t,\tau)\BB_\tau(R)\subset\B_t.$$
\end{definition}

The existence of a time-dependent absorbing set (hence pullback absorbing) for $U(t,\tau)$ is witnessed by

\begin{theorem}
\label{t-abset}
There exists $R_0>0$ such that the family ${\mathfrak B}=\{\BB_t(R_0)\}_{t\in\R}$
is a time-dependent absorbing set for $U(t,\tau)$.
Besides,
\begin{equation}
\label{diss-int}
\sup_{z\in \BB_\tau(R_0)}\Big[\|U(t,\tau)z\|_{\H_t}
+\int_\tau^\infty\|u_t(y)\|^2\,\d y\Big]\leq I_0,\quad \forall \tau\in\R,
\end{equation}
for some $I_0\geq R_0$.
\end{theorem}

As already discussed, we propose an easy and direct proof of this result,
based on the extra assumption \eqref{extra}. The crucial ingredient is the following dissipation estimate.

\begin{lemma}
\label{l-energy0}
Let $t\geq \tau$. For $z\in \H_\tau$, let $U(t,\tau)z$ be the solution of
\eqref{P} with initial time $\tau$ and datum $z=\{a,b\}$.
Then, if \eqref{extra} holds, there exist $\omega=\omega(\alpha,\|\eps\|_{L^\infty},\|\eps'\|_{L^\infty})>0$, $K_1\geq0$
and an increasing positive function $\Q$ such that
$$\|U(t,\tau)z\|_{\H_t}\leq \Q(\|z\|_{\H_\tau})\e^{-\omega(t-\tau)}+K_1,\quad \forall \tau\leq t.$$
\end{lemma}

\begin{proof}
Let $C\geq0$ be a {\it generic}
constant independent
of the initial datum $z$ and denote
$$E(t)=\|U(t,\tau)z\|_{\H_t}^2$$
(double) the {\it energy} associated with problem \eqref{P}.
Due to~\eqref{e2}, \eqref{phi2}
and \eqref{phi-funz1}, the functional
\begin{align*}
\E = E+\delta\alpha \|u\|^2+2\delta\eps\l u_t,u\r+2\l F(u),1\r-2\l g,u\r
\end{align*}
fulfills, for $\delta > 0$ small and some $0<\nu<1$ provided by Lemma \ref{l-tec2},
\begin{equation}
\label{equivEnergia}
\nu E(t) - C \leq \E(t) \leq CE^2(t)+C.
\end{equation}
Indeed, in light of \eqref{e2}, if $\delta$  is small enough we have
$$2\delta\eps|\l u_t,u \r|\leq
\frac{\eps}{2}\|u_t\|^2+CL\delta^2\|u\|^2\leq \frac{\eps}{2}\|u_t\|^2+\frac{\delta\alpha}{2}\|u\|^2.$$
Multiplying ~\eqref{EQ} by $2u_t+2\delta u$, we infer
\begin{align*}
%\label{peri}
\ddt \E +&[2\alpha-\eps'-2\delta\eps]\|u_t\|^2+2\delta
\|u\|_1^2 +2\delta\l f(u),u\r-2\delta\l g,u\r=
2\delta\eps'\l u_t,u \r,
\end{align*}
and estimating
$$2\delta|\eps'\l u_t,u \r|\leq 2\delta L\|u_t\|\|u\|\leq \frac{\alpha}{2}\|u_t\|^2+\frac{\delta\nu}{2}\|u\|_1^2$$
for $\delta$ small, we arrive at
\begin{equation}
\label{e0}
\ddt \E +\Big[\frac{3}{2}\alpha-\eps'-2\delta\eps\Big]\|u_t\|^2 +\delta\Big[2-\frac{\nu}{2}\Big]\|u\|_1^2
+2\delta\l f(u),u\r-2\delta\l g,u\r\leq 0.
\end{equation}
In light of ~\eqref{extra} we can reconstruct the functional $\E$, which provides
\begin{equation*}
\ddt \E +\delta\E +\alpha\|u_t\|^2+\Gamma
\leq \delta c_1,
\end{equation*}
where
$$\Gamma=\Big[\frac{\alpha}{2}-\eps'-3\delta\eps\Big]\|u_t\|^2+\frac{\delta\nu}{2}
\|u\|_1^2-\delta^2\alpha\|u\|^2-2\delta^2\eps\l u_t,u\r.$$
Therefore, setting $\delta$ small enough so that $\Gamma\geq 0$,
we end up with
\begin{equation}
\label{diff-in1}
\ddt \E + \delta \E + \alpha\|u_t\|^2 \leq \delta c_1.
\end{equation}
Applying the Gronwall lemma, together with~\eqref{equivEnergia}, we have proved Lemma \ref{l-energy0}.
\end{proof}

\medskip
\noindent
{\it Proof of Theorem \ref{t-abset}.}
Let $R_0=1+2K_1$. An application of Lemma \ref{l-energy0} for $z\in \BB_\tau(R)$ yields
$$\|U(t,\tau)z\|_{\H_t}\leq \Q(R)\e^{-\omega(t-\tau)}+K_1\leq 1+2K_1=R_0,$$
provided that $t-\tau\geq \theta_\e$, where
$$\theta_\e=\max\Big\{0, \omega^{-1}\log\frac{\Q(R)}{1+K_1}\Big\}.$$
This concludes the proof of the existence of the time-dependent absorbing set.
In order to prove the integral estimate for $\|u_t\|$, it is enough to integrate
\eqref{diff-in1} with $\delta=0$ on $[\tau,\infty)$.

\begin{remark}
We can assume that the time-dependent absorbing set $\B_t=\BB_t(R_0)$  is positively invariant (namely
$U(t,\tau)\B_\tau  \subset \B_t$ for all $t\geq \tau$). Indeed, calling $\theta_\e$ the entering time of $\B_t$
such that
$$U(t,\tau)\B_\tau\subset \B_t,\quad \forall\tau\leq t-\theta_\e,$$
we can substitute $\B_t$ with the invariant absorbing family
$$\bigcup_{\tau\leq t-\theta_\e}U(t,\tau)\B_\tau\subset \B_t.$$
\end{remark}
%%%%%%%%%%%%%%%%%%%%%%%%%%%%%%%%%%%%%%%%%%%%%%%%%

%%%%%%%%%%%%%%%%%%%%%%%%%%%%%%%%%%%%%%%%%%%%%%%%%
\section{Existence of the Time-Dependent Global Attractor}

\noindent
The main result concerning the asymptotic behavior of problem \eqref{P} is contained in the
following theorem.

\begin{theorem}
\label{MAIN-WAVE}
The process $U(t,\tau):\H_\tau\to\H_t$ generated by problem \eqref{P} admits
an invariant time-dependent global attractor $\AA=\{\A_t\}_{t\in\R}$. Besides,
$\A_t$ is bounded in $\H_t^1$, with a bound independent of $t$.
\end{theorem}

The existence of the attractor, according to Definition \ref{d-attractor},
will be proved by a direct application of the abstract
Theorem \ref{MAA}. Precisely, in order to show that
the process is asymptotically compact,
we shall exhibit a
pullback attracting family of (nonvoid) compact sets.
To this aim, the strategy classically consists in finding a suitable decomposition of the
process in the sum of a {\it decaying} part and of a {\it compact} one.

\subsection{The Decomposition}
We write $f=f_0+f_1$,
where $f_0,f_1\in C^2(\R)$ fulfill, for some $k\geq 0$,
\begin{align}
\label{phiuno}
& |f'_1(s)|\leq k,\quad\forall s\in\R,\\
\label{crez}
& |f''_0(s)|\leq k(1+|s|),\quad
\forall s\in\R,\\
\label{atzero}
&f_0(0)=f_0'(0)=0,\\
\label{mono}
& f_0(s)s\geq 0,\quad\forall s\in\R.
\end{align}
This is possible owing to assumptions \eqref{phi2} and \eqref{phi1} (cf.\ \cite{ach,GP-Grand}).

\medskip
\noindent Let ${\mathfrak B}=\{\BB_t(R_0)\}_{t\in\R}$ be a time-dependent absorbing set according to
Theorem \ref{t-abset}
and let $\tau\in\R$  be fixed. Then, for any $z\in\BB_\tau(R_0)$, we split $U(t, \tau)z$ into the sum
$$U(t,\tau)z=\{u(t),u_t(t)\}=U_0(t,\tau)z+U_1(t,\tau)z,$$
where
$$U_0(t,\tau)z=\{v(t), v_t(t)\}\and
U_1(t,\tau)z=\{w(t), w_t(t)\}$$
solve the systems
\begin{equation}
\label{PL}
\begin{cases}
\eps v_{tt}+ \alpha v_t+Av+f_0(v)=0,\\
U_0(\tau,\tau)=z,
\end{cases}
\end{equation}
and
\begin{equation}
\label{PN}
\begin{cases}
\eps w_{tt}+ \alpha w_t+Aw+f(u)-f_0(v)=g,\\
U_1(\tau,\tau)=0.
\end{cases}
\end{equation}
In what follows, the {\it generic} constant $C\geq 0$ depends
only on $\mathfrak B$.

\begin{lemma}\label{DECAY}
There exists $\delta=\delta({\mathfrak B})>0$ such that
$$
\|U_0(t,\tau)z\|_{\H_t}\leq C\e^{-\delta(t-\tau)},\quad\forall t\geq\tau.
$$
\end{lemma}

\begin{proof}
Repeating word by word the proof of Lemma \ref{l-energy0} with $f_0$
instead of $f$ we immediately get the bound
\begin{equation}
\label{vbound}
\|U_0(t,\tau)z\|_{\H_t}\leq C.
\end{equation}
Then, denoting
\begin{align*}
\E_0 =& \|U_0(t,\tau)z\|_{\H_t}^2+\delta\alpha \|v\|^2+2\delta\eps\l v_t,v\r+2\l F_0(v),1\r,
\end{align*}
with
$$F_0(s)=\int_0^s f_0(y)\,\d y,$$
we multiply \eqref{PL} by $2v_t+2\delta v$.
In view of \eqref{mono} and since $g=0$,
the analogous of the differential inequality \eqref{e0} now reads
\begin{equation*}
\ddt \E_0 + \delta\|U_0(t,\tau)z\|_{\H_t}^2\leq 0.
\end{equation*}
Exploiting \eqref{vbound} we have
$$\frac12\|U_0(t,\tau)z\|_{\H_t}^2\leq \E_0(t)\leq C\|U_0(t,\tau)z\|_{\H_t}^2,$$
and the Gronwall lemma completes the argument.
\end{proof}

Summing up, the following uniform bound holds
\begin{equation}
\label{TTRI}
\sup_{t\geq \tau}\big[\|U(t,\tau)z\|_\Ht+\|U_0(t,\tau)z\|_\Ht+\|U_1(t,\tau)z\|_\Ht\big]\leq C.
\end{equation}

\begin{lemma}
\label{unterzo}
There exists $M=M({\mathfrak B})>0$ such that
$$
\sup_{t\geq \tau}\|U_1(t,\tau)z)\|_{\H_t^{1/3}}\leq M.
$$
\end{lemma}

\begin{proof}
We choose $\delta>0$ small and $C>0$ large enough such that,
calling
$$
\Lambda=\|U_1(t,\tau)z\|_{\H_t^{1/3}}^2+\delta\alpha\|w\|_{{1/3}}^2
+2\delta\eps\l w_t,A^{{1/3}}w\r+2\langle f(u)-f_0(v)-g,A^{1/3}w\rangle+C,$$
we have
\begin{equation}
\label{DYLAN}
\frac12 \|U_1(t,\tau)z\|_{\H_t^{1/3}}^2\leq \Lambda(t)\leq 2\|U_1(t,\tau)z)\|_{\H_t^{1/3}}^2+2C.
\end{equation}
Indeed, in view of \eqref{TTRI} and the growth of $f$,
\begin{align*}
2\langle f(u)-f_0(v),A^{1/3}w\rangle&\leq 2\|f(u)-f_0(v)\|\|A^{1/3}w\|\leq C\|w\|_{2/3}\leq \frac14\|w\|_{4/3}^2+C.
\end{align*}
Besides, by \eqref{e2}, for $\delta$  small we can estimate
$$2\delta\eps|\l w_t,A^{{1/3}}w\r|\leq  \frac{\eps}{2}\|w_t\|_{1/3}^2+\frac{\delta\alpha}{2}\|w\|_{1/3}^2.$$
By multiplying \eqref{PN} with $2A^{1/3}w_t+2\delta A^{1/3}w$, we infer that
\begin{align*}
\ddt \Lambda&+[2\alpha-\eps'-2\delta\eps]\|w_t\|_{1/3}^2+2\delta
\|w\|_{4/3}^2 +2\delta\l f(u)-f_0(v)-g,A^{1/3}w\r\\&=
2\delta\eps'\l w_t,A^{1/3}w \r+I_1+I_2+I_3,
\end{align*}
where
\begin{align*}
I_1&=2\l [f_0'(u)-f_0'(v)]u_t,A^{1/3}w\r,\\
I_2&=2\l f_0'(v) w_t,A^{1/3}w\r,\\
I_3&=2\l f_1'(u)u_t ,A^{1/3}w\r.\\
\end{align*}
Then, for any fixed $\delta>0$ small enough,
we easily get
\begin{equation}
\label{QQ1}
\ddt \Lambda+\delta \Lambda+\alpha\|w_t\|_{1/3}^2
\leq I_1+I_2+I_3+\delta C.
\end{equation}
By exploiting conditions \eqref{crez}-\eqref{atzero} for $f_0$ and the embeddings
$\HH_{(3p-6)/2p}\subset L^{p}(\Omega)$ ($p>2$),
we draw from~\eqref{TTRI}-\eqref{DYLAN} the estimates
\begin{align*}
I_1&\leq
C\big(1+\|u\|_{L^6}+\|v\|_{L^6}\big)
\|u_t\|\|w\|_{L^{18}}\|A^{1/3}w\|_{L^{18/5}}
\leq C\|u_t\|\|w\|_{4/3}^2\\
&\leq\frac{\delta}{2}\Lambda+C\|u_t\|^2\|w\|_{4/3}^2,\\
\smallskip
I_2&\leq
C\big(\|v\|_{L^6}+\|v\|_{L^6}^2\big)\| w_t\|_{L^{18/7}}
\|A^{1/3} w\|_{L^{18/5}}\leq C\|v\|_1\| w_t\|_{1/3}\|w\|_{4/3}\\
&\leq\frac{\alpha}{2}\|w_t\|_{1/3}^2+C\|v\|_1^2 \|w\|_{4/3}^2.
\end{align*}
Besides, in view of \eqref{phiuno}, we have
$$I_3\leq k\|u_t\|\|A^{1/3}w\|\leq \|u_t\|^2\|w\|_{4/3}^2+C.$$
As a consequence, inequality~\eqref{QQ1} improves to
$$
\ddt \Lambda+\frac{\delta}{2} \Lambda
\leq q \Lambda+C,
$$
with
$q=C\|u_t\|^2+C\|v\|_1^2$ satisfying
$$
\int_{\tau}^{\infty}q(y)\,\d y\leq C,
$$
by virtue of the dissipation integral \eqref{diss-int} and Lemma~\ref{DECAY}.
By Lemma \ref{gronwall},
$$
\Lambda(t)\leq C\Lambda(\tau)\e^{-\frac{\delta}{4} (t-\tau)}+C\leq C.
$$
In turn, \eqref{DYLAN}
yields the boundedness of $U_1(t,\tau)z$ in $\H_t^{1/3}$.
\end{proof}

\subsection{Existence of the invariant attractor}
According to Lemma \ref{unterzo},  we consider the family $\mathfrak K=\{K_t\}_{t\in\R}$
where
$$K_t=\Big\{z\in\H_t^{1/3}: \|z\|_{\H^{1/3}_t}\leq M\Big\}.$$
$K_t$ is compact by the compact embedding $\H^{1/3}_t\Subset\H_t$; besides,
since the injection constants are independent of $t$,
$\mathfrak K$ is uniformly bounded.
Finally, Theorem \ref{t-abset}, Lemma \ref{DECAY} and Lemma \ref{unterzo} show that $\mathfrak K$ is pullback attracting;
indeed,
$$\di_t(U(t,\tau) \BB_\tau(R_0),\K_t)\leq C\e^{-\delta(t-\tau)},\quad\forall t\geq\tau.$$
Hence the process $U(t,\tau)$ is asymptotically compact, which allows
the application of Theorem \ref{MAA} and proves
the existence of the unique time-dependent global attractor $\AA=\{\A_t\}_{t\in\R}$.
The invariance of $\AA$ follows by the abstract Theorem \ref{invariance-2}, due to
the strong continuity of the process stated in Theorem \ref{wellposed}.

\begin{remark}
The attraction exerted by the attractor is uniform on compact intervals of time
by virtue of Proposition \ref{unif-att}, due to the continuous dependence estimate  \eqref{contdep}.
\end{remark}

\subsection{Regularity of the attractor}
The minimality of $\AA$ in $\KK$
establishes that $\A_t\subset K_t$ for all $t\in\R$. Therefore,
we immediately have the following regularity result.
\begin{corollary}
\label{corBdd}
$\A_t$ is bounded in $\H^{1/3}_t$  (with a bound independent of $t$).
\end{corollary}
To prove that $\A_t$ is bounded in $\H^1_t$, as claimed in Theorem \ref{MAIN-WAVE}, we argue as follows.
We fix $\tau\in\R$ and, for $z\in\A_\tau$,
we split  the solution $U(t,\tau)z$
into the sum $U_0(t,\tau)z+U_1(t,\tau)z$,
where $U_0(t,\tau)z=\{v(t), v_t(t)\}$ and $U_1(t,\tau)z=\{w(t), w_t(t)\}$, instead of \eqref{PL}-\eqref{PN},
now solve
$$
\begin{cases}
\displaystyle
\eps v_{tt}+\alpha v_t+Av=0,\\
U_0(\tau,\tau)=z,
\end{cases}
\quad
\begin{cases}
\displaystyle
\displaystyle
\eps w_{tt}+\alpha w_t+Aw+f(u)=g,\\
U_1(\tau,\tau)=0.
\end{cases}
$$
As a particular case of Lemma~\ref{DECAY}, we learn that
\begin{equation}\label{decay-UNO}
\|U_0(t,\tau)z\|_{\H_t}\leq C\e^{-\delta(t-\tau)},\quad\forall t\geq\tau.
\end{equation}

\begin{lemma}
\label{bd-UNO}
We have the uniform bound
$$
\sup_{t\geq \tau}\|U_1(t,\tau)z\|_{\H^1_t}\leq M_1,
$$
for some $M_1=M_1({\mathfrak A})>0$.
\end{lemma}

\begin{proof}
We set
$$
\E_1=\|U_1(t,\tau)z\|_{\H^1_t}^2+\delta\alpha\|w\|_{{1}}^2
+2\delta\eps\l w_t,A w\r-2\l g,Aw\r+c,$$
for $\delta>0$ small and some $c\geq0$ (depending on $\|g\|$) large enough to ensure
\begin{equation}
\label{DYLAN1}
\frac14 \|U_1(t,\tau)z\|_{\H^1_t}^2\leq \E_1(t)\leq 2 \|U_1(t,\tau)z\|_{\H^1_t}^2+2c.
\end{equation}
A multiplication by $2Aw_t+2\delta Aw$
leads to the equality
\begin{align*}
\ddt \E_1 &+
[2\alpha-\eps'-2\delta\eps]\|w_t\|_{1}^2+2\delta
\|w\|_{2}^2-2\delta\l g,Aw\r\\& =
2\delta\eps'\l w_t,Aw \r
-2\l f(u), Aw_t\r
-2\delta \l f(u),Aw\r,
\end{align*}
and after standard computations we get, for $\delta$ small enough,
\begin{align*}
\ddt \E_1+\delta\E_1 &\leq -2\l f(u), Aw_t\r
-2\delta \l f(u),Aw\r+\delta c.
\end{align*}
Denoting by $C>0$ a {\it generic} constant depending on the size of $\A_t$ in $\H^{1/3}_t$,
we find, using the invariance of the attractor,
$$\|U(t,\tau)z\|_{\H^{1/3}_t}\leq C.$$
Hence, exploiting the embeddings
$\HH_{4/3}\subset L^{18}(\Omega)$ and $\HH_{1/3}\subset L^{18/7}(\Omega)$,
we deduce the bound
$$
\|f(u)\|_1\leq\|f'(u)\|_{L^9}\|A^{1/2}u\|_{L^{18/7}}
\leq C\big(1+\|u\|_{L^{18}}^2\big)\leq C,
$$
yielding
\begin{align*}
-2\l f(u), Aw_t\r-2\delta\l f(u),Aw\r\leq 2\|f(u)\|_1(\|w_t\|_1+\|w\|_1)\leq \frac{\delta}{2}\E_1+C.
\end{align*}
We finally end up with
$$\ddt \E_1+\frac{\delta}{2} \E_1\leq C,$$
and an application of the standard Gronwall lemma,
recalling \eqref{DYLAN1}, provides the uniform boundedness of $\|U_1(t,\tau)z\|_{\H^1_t}$, as claimed.
\end{proof}
We are now in position to conclude the proof of Theorem \ref{MAIN-WAVE}.
Indeed, inequality \eqref{decay-UNO} and Lemma \ref{bd-UNO} imply that, for all $t\in\R$,
$$\lim_{\tau\to-\infty}\di_t(U(t,\tau)\A_\tau, K^1_t)=0,$$
having defined
$$K_t^{1}=\Big\{z\in\H_t^{1}: \|z\|_{\H^1_t}\leq M_1\Big\}.$$
Since $\AA$ is invariant, this means
$$\di_t(\A_t, K^1_t)=0.$$
Hence, $\A_t\subset\overline{K_t^1}=K_t^1$, proving that $\A_t$
is bounded in $\H_t^{1}$ with a bound independent of $t\in\R$.
%%%%%%%%%%%%%%%%%%%%%%%%%%%%%%%%%%%%%%%%%%%%%%%%%%%%%%%%%%%%

\subsection*{Acknowledgments} This work was partially supported by the National Science
Foundation under the grants NSF-DMS-1206438, NSF-DMS-0906440, and by the
Research Fund of Indiana University.

%%%%%%%%%%%%%%%%%%%%%%%%%%%%%%%%%%%%%%%%%%%%%%%%%%%%%%%%%%%%

%%%%%%%%%%%%%%%%%%%%%%%%%%%%%%%%%%%%%%%%%%%%%%%%%%%%%%%%%%%%

\end{document}